\pdfoutput=1
\documentclass[12pt]{article}
\usepackage[latin1]{inputenc}
\usepackage[british]{babel}
\usepackage{cmap}
\usepackage{lmodern}

\usepackage{amssymb, amsmath, amsthm}
\usepackage[a4paper,top=25mm,bottom=25mm,left=25mm,right=25mm]{geometry}
\usepackage{ragged2e}

\usepackage{authblk} 
\usepackage{pifont}
\usepackage{graphicx}
\usepackage[usenames,dvipsnames,svgnames,table]{xcolor}
\usepackage[figuresright]{rotating}
\usepackage{xtab} 
\usepackage{longtable} 
\usepackage{multirow}
\usepackage{footnote}
\usepackage[stable]{footmisc}
\usepackage{chngpage} 
\usepackage{pdflscape} 
\usepackage[nottoc,notlot,notlof]{tocbibind} 

\usepackage{pgfplots}
\pgfplotsset{every tick label/.append style={font=\footnotesize}}
\pgfplotsset{compat=1.18}
\usepackage{setspace}

\usepackage{array}
\newcolumntype{K}[1]{>{\centering\arraybackslash$}p{#1}<{$}}
\usepackage{algorithm}
\usepackage{algpseudocode}

\makesavenoteenv{tabular}
\usepackage{tabularx}
\usepackage{booktabs}
\usepackage{threeparttable}
\usepackage[referable]{threeparttablex} 
\newcolumntype{R}{>{\raggedleft\arraybackslash}X}
\newcolumntype{L}{>{\raggedright\arraybackslash}X}
\newcolumntype{C}{>{\centering\arraybackslash}X}
\newcolumntype{A}{>{\columncolor{gray!25}}C}
\newcolumntype{a}{>{\columncolor{gray!25}}c}

\newlength{\tablen}

\usepackage{dcolumn} 
\newcolumntype{.}{D{.}{.}{-1}}

\usepackage{tikz}
\usetikzlibrary{arrows, calc, matrix, patterns, positioning, shapes, trees}
\usepackage[semicolon]{natbib}
\usepackage[hyphens]{url}
\usepackage{hyperref} 
\hypersetup{
  colorlinks   = true,    		
  urlcolor     = blue,    		
  linkcolor    = blue,    		
  citecolor    = ForestGreen	
}
\usepackage{microtype}
\usepackage[justification=centering]{caption} 

\usepackage[labelformat=simple]{subcaption}

\DeclareCaptionLabelFormat{parenthesis}{(#2)}
\makeatletter
\renewcommand\p@subfigure{\arabic{figure}.}
\makeatother

\DeclareCaptionLabelFormat{parenthesis}{(#2)}
\makeatletter
\renewcommand\p@subtable{\arabic{table}.}
\makeatother

\usepackage{enumitem}

\setlist[itemize]{leftmargin=2.5\parindent}
\setlist[enumerate]{leftmargin=2.5\parindent}

\def\addlegendimage{\csname pgfplots@addlegendimage\endcsname}

\theoremstyle{plain}

\newtheorem{lemma}{Lemma}

\newtheorem{theorem}{Theorem}

\theoremstyle{definition}

\newtheorem{definition}{Definition}
\newtheorem{example}{Example}

\theoremstyle{remark}

\makeatletter
\let\@fnsymbol\@alph
\makeatother

\def\keywords{\vspace{.5em} 
{\noindent \textit{Keywords}: }}

\def\AMS{\vspace{.5em} 
{\noindent \textbf{\emph{MSC} class}: }}

\def\JEL{\vspace{.5em} 
{\noindent \textbf{\emph{JEL} classification number}: }}

\title{How to choose a completion method for pairwise comparison matrices with missing entries: \\ An axiomatic result}
\author{\href{https://sites.google.com/view/laszlocsato}{L\'aszl\'o Csat\'o}\thanks{~E-mail: \emph{laszlo.csato@sztaki.hu}} }
\affil{Institute for Computer Science and Control (SZTAKI) \\
E\"otv\"os Lor\'and Research Network (ELKH) \\
Laboratory on Engineering and Management Intelligence \\
Research Group of Operations Research and Decision Systems}
\affil{Corvinus University of Budapest (BCE) \\
Institute of Operations and Decision Sciences \\
Department of Operations Research and Actuarial Sciences}
\affil{Budapest, Hungary}
\date{\today}

\def\Dedication{
\begin{small}
{\noindent
``\emph{An important property of any weighting method is the ability to preserve the ordinal preferences which are implicitly expressed by the ratio-scale preference matrix entries.}''\footnote{~Source: \citet[p.~214]{GolanyKress1993}.}
}
\end{small}

\vspace{0.5cm} 
\justify }

\begin{document}

\maketitle
\thispagestyle{empty}
\Dedication

\begin{abstract}
\noindent
Since there exist several completion methods to estimate the missing entries of pairwise comparison matrices, practitioners face a difficult task in choosing the best technique. Our paper contributes to this issue: we consider a special set of incomplete pairwise comparison matrices that can be represented by a weakly connected directed acyclic graph, and study whether the derived weights are consistent with the partial order implied by the underlying graph.
According to previous results from the literature, two popular procedures, the incomplete eigenvector and the incomplete logarithmic least squares methods fail to satisfy the required property. Here, the recently introduced lexicographically optimal completion combined with any of these weighting methods is shown to avoid ordinal violation in the above setting. Our finding provides a powerful argument for using the lexicographically optimal completion to determine the missing elements in an incomplete pairwise comparison matrix.


\AMS{90B50, 91B08}

\JEL{C44, D71}
\end{abstract}

\clearpage

\section{Introduction} \label{Sec1}

Decision theory extensively uses pairwise comparisons. For example, the popular AHP (Analytic Hierarchy Process) methodology \citep{Saaty1977, Saaty1980} establishes priorities among the alternatives and criteria based on their pairwise comparisons.
Therefore, a fundamental element of the decision-making process is deriving priorities from the pairwise comparison matrix. Several procedures exist to that end \citep{ChooWedley2004}, the most common choices being the eigenvector \citep{Saaty1977} and the logarithmic least squares/row geometric mean \citep{CrawfordWilliams1985} methods.

However, some pairwise comparisons may be missing due to the lack of data or the inability of an expert to compare two alternatives. This is often the case in sporting contexts if some players do not play against each other \citep{BozokiCsatoTemesi2016, ChaoKouLiPeng2018, Csato2013a, TemesiSzadoczkiBozoki2023}. Fortunately, both the eigenvector and the logarithmic least squares methods have been extended to incomplete pairwise comparison matrices \citep{BozokiFulopRonyai2010}, and there are several other completion techniques proposed in the literature \citep{TekileBrunelliFedrizzi2023}.

Nonetheless, practitioners may face a dilemma when they should work with incomplete pairwise comparison matrices: Which completion method should be chosen? In order to contribute to this issue, we take an axiomatic approach. 
In particular, the current paper considers incomplete pairwise comparison matrices generated by (connected) directed acyclic graphs. and set up a natural research question: Is there a pair of a completion and a weighting method to obtain priorities that are certainly free from any ordinal violation?

The answer is far from trivial because both the eigenvector and logarithmic least squares solutions can yield a ranking that contradicts the ordinally consistent preferences if the number of items is at least seven \citep{CsatoRonyai2016, FaramondiOlivaBozoki2020}.
Here, the recently proposed \emph{lexicographically optimal completion} \citep{AgostonCsato2023} is proved to result in no ordinal violation if
(1) the preferences are represented by a directed acyclic graph, and
(2) the weight vector is determined by the eigenvector method or the logarithmic least squares method.
Our main contribution resides in finding a procedure for pairwise comparison matrices with missing entries to guarantee a reasonable ranking without any additional restrictions. 

The remainder of the paper is organised as follows. Section~\ref{Sec2} introduces incomplete pairwise comparison matrices and their optimal completions. The lack of ordinal violations is verified in Section~\ref{Sec3} for the above procedure. Finally, Section~\ref{Sec4} offers concluding remarks.

\section{Preliminaries} \label{Sec2}

In the following, we concisely present all necessary definitions and notations.

\subsection{Pairwise comparison matrices and weighting methods} \label{Sec21}

Denote by $\mathbb{R}^n_+$ the set of positive vectors of size $n$, and by $\mathbb{R}^{n \times n}_+$ the set of $n \times n$ positive matrices.

\begin{definition} \label{Def1}
\emph{Pairwise comparison matrix}:
Matrix $\mathbf{A} = \left[ a_{ij} \right] \in \mathbb{R}^{n \times n}_+$ is a \emph{pairwise comparison matrix} if $a_{ji} = 1 / a_{ij}$ for all $1 \leq i,j \leq n$.
\end{definition}

Throughout the paper, the set of pairwise comparison matrices is denoted by $\mathcal{A}$, and the set of $n \times n$ pairwise comparison matrices is denoted by $\mathcal{A}^{n \times n}$.

\begin{definition} \label{Def2}
\emph{Consistency}:
A pairwise comparison matrix $\mathbf{A} = \left[ a_{ij} \right] \in \mathcal{A}^{n \times n}$ is called \emph{consistent} if $a_{ik} = a_{ij} a_{jk}$ holds for all $1 \leq i,j,k \leq n$.
\end{definition}

In practice, pairwise comparison matrices are usually \emph{inconsistent}.
The level of inconsistency can be quantified by inconsistency indices, see \citet{Brunelli2018} for a survey.

\begin{definition} \label{Def3}
\emph{Weighting method}:
A \emph{weighting method} associates a weight vector $\mathbf{w} = \left[ w_i \right] \in \mathbb{R}^n_+$ satisfying $\sum_{i=1}^n w_i = 1$ to any pairwise comparison matrix $\mathbf{A} = \left[ a_{ij} \right] \in \mathcal{A}^{n \times n}$.
\end{definition}

The most popular weighting methods are as follows.

\begin{definition} \label{Def4}
\emph{Eigenvector method} \citep{Saaty1977, Saaty1980}:
Let $\mathbf{A} = \left[ a_{ij} \right] \in \mathcal{A}^{n \times n}$ be a pairwise comparison matrix. The weight vector $\mathbf{w} = \left[ w_i \right] \in \mathbb{R}^n_+$ provided by the \emph{eigenvector method} is the solution $\mathbf{w}$ of a system of linear equations:
\begin{equation} \label{eq_EM}
\lambda_{\max} \mathbf{w} = \mathbf{A} \mathbf{w},
\end{equation}
that is, $\lambda_{\max} \mathbf{A}$ is the dominant eigenvalue of matrix $\mathbf{A}$ and $\mathbf{w}$ is the associated right eigenvector.
\end{definition}

\begin{definition} \label{Def5}
\emph{Logarithmic least squares method} \citep{CrawfordWilliams1985, DeGraan1980, deJong1984, Rabinowitz1976, WilliamsCrawford1980}:
Let $\mathbf{A} = \left[ a_{ij} \right] \in \mathcal{A}^{n \times n}$ be a pairwise comparison matrix. The weight vector $\mathbf{w} = \left[ w_i \right] \in \mathbb{R}^n_+$ provided by the \emph{logarithmic least squares method} is the solution $\mathbf{w}$ of the following optimisation problem:
\begin{align} \label{eq_LLSM}
\min & \sum_{i=1}^n \sum_{j=1}^n \left[ \log a_{ij} - \log \left( \frac{w_i}{w_j} \right) \right]^2 \nonumber \\
\text{subject to } & w_i > 0 \text{ for all } i=1,2, \dots n.
\end{align}
\end{definition}

It can be shown that the unique solution $\mathbf{w}$ of \eqref{eq_LLSM} (up to multiplication by a positive constant) is given by the geometric means of row elements, namely,
\begin{equation} \label{eq_RGM}
w_i = \sqrt[n]{\prod_{i=1}^n a_{ij}} \qquad \text{for all} \qquad 1 \leq i \leq n.
\end{equation}

\citet{ChooWedley2004} discuss several other weighting methods. 

\subsection{Directed graphs and incomplete pairwise comparisons} \label{Sec22}

Ordinal pairwise comparisons can be represented by a directed graph.

\begin{definition} \label{Def6}
\emph{Directed graph}:
The tuple $(N,E)$ is a \emph{directed graph}, where $N$ is the set of vertices and $E$ is the set of ordered pairs of vertices (arcs).
\end{definition}

\begin{definition} \label{Def7}
\emph{Directed walk}:
Let $(N,E)$ be a directed graph.
A \emph{directed walk} is a sequence of arcs in $E$ such that the ending vertex of each arc in the sequence is the same as the starting vertex of the next arc in the sequence.
\end{definition}

\begin{definition} \label{Def8}
\emph{Cycle}:
Let $(N,E)$ be a directed graph.
A directed walk is called a \emph{cycle} if the starting vertex of the first arc coincides with the ending vertex of the last arc.
\end{definition}

\begin{definition} \label{Def9}
\emph{Connected directed acyclic graph} (CDAG):
A directed graph $(N,E)$ is \emph{acyclic} if it does not contain any cycle. \\
A directed graph $(N,E)$ is \emph{(weakly) connected} if the underlying undirected graph, where all arcs are replaced by undirected edges, is connected.
\end{definition}

An incomplete pairwise comparison matrix may contain missing entries outside its diagonal, which are denoted by $\ast$. 

\begin{definition} \label{Def10}
\emph{Incomplete pairwise comparison matrix}:
Pairwise comparison matrix $\mathbf{A} = \left[ a_{ij} \right] \in \mathcal{A}^{n \times n}$ is an \emph{incomplete pairwise comparison matrix} if  $a_{ij} > 0$ implies $a_{ji} = 1 / a_{ij}$ and $a_{ij} = \ast$ implies $a_{ji} = \ast$ for all $1 \leq i,j \leq n$.
\end{definition}

The set of incomplete pairwise comparison matrices is denoted by $\mathcal{A}^\ast$.

\begin{definition} \label{Def11}
\emph{CDAG-based incomplete pairwise comparison matrix}:
Let $(N,E)$ be a connected directed acyclic graph with $|N| = n$.
Matrix $\mathbf{A} = \left[ a_{ij} \right] \in \mathcal{A}^\ast$ is a \emph{CDAG-based incomplete pairwise comparison matrix} if $a_{ij} \in \{ 1 / \alpha;\, 1;\, \alpha;\, \ast \}$ such that for all $1 \leq i,j \leq n$:
\begin{itemize}
\item
$a_{ii} = 1$;
\item
$(i,j) \in E$ implies $a_{ij} = \alpha > 1$ and $a_{ji} = 1 / \alpha < 1$;
\item
$(i,j) \notin E$ and $(j,i) \notin E$ imply $a_{ij} = \ast$ and $a_{ji} = \ast$.
\end{itemize}
\end{definition}

To summarise, the ordinal preferences contained in the connected directed acyclic graph are represented by numerical values that satisfy the reciprocity property of pairwise comparison matrices.
Obviously, there exists a one-to-one correspondence between directed acyclic graphs and incomplete pairwise comparison matrices for which the conditions of Definition~\ref{Def11} hold.

\subsection{Optimal completions of missing pairwise comparisons} \label{Sec23}

The missing entries of incomplete pairwise comparison matrices are usually determined by replacing them with numerical values that minimise an inconsistency index \citep{KoczkodajHermanOrlowski1999}.
Perhaps the most popular inconsistency measure $CR$ has been suggested by \citet{Saaty1977, Saaty1980}. Since it is a linear transformation of the dominant eigenvalue $\lambda_{\max}$, \citet{ShiraishiObataDaigo1998} and \citet{ShiraishiObata2002} have proposed to choose the missing values in order to minimise $\lambda_{\max}$. \citet{BozokiFulopRonyai2010} have solved the corresponding optimisation problem that leads to the \emph{$CR$-optimal completion} of the incomplete pairwise comparison matrix.

Another widely used inconsistency measure is the geometric consistency index $GCI$ \citep{CrawfordWilliams1985, AguaronMoreno-Jimenez2003}. Minimising $GCI$ for incomplete pairwise comparison matrices yields the incomplete logarithmic least squares method \citep{BozokiFulopRonyai2010, BozokiTsyganok2019}. Although this technique provides only the optimal weight vector directly, the \emph{$GCI$-optimal completion} can be easily obtained by replacing each missing entry $a_{ij}$ with the corresponding ratio $w_i / w_j$ of the optimal weights.

According to \citet{BozokiFulopRonyai2010}, both the $CR$- and $GCI$-optimal completions are unique if and only if the graph associated with the incomplete pairwise comparison is connected, i.e.\ any two items can be compared at least indirectly through other items.

In every inconsistent pairwise comparison matrix, at least one inconsistent triad exists. There is only one reasonable measure of triad inconsistency \citep{Csato2019b, Cavallo2020}:
\begin{equation} \label{eq_TI}
TI = \max \left\{ \frac{a_{ik}}{a_{ij} a_{jk}}; \frac{a_{ij} a_{jk}}{a_{ik}} \right\}.
\end{equation}
In particular, note that \citet{BozokiRapcsak2008} have proved a (monotonic) functional relationship between the inconsistency index $TI$, the Koczkodaj inconsistency index $KI$ \citep{Koczkodaj1993, DuszakKoczkodaj1994}, and Saaty's inconsistency ratio $CR$ on the set of triads. Furthermore, according to \cite{Cavallo2020}, almost all inconsistency indices---including the geometric inconsistency index $GCI$---are functionally dependent for $n = 3$. \citet{Csato2019b} has given an axiomatic characterisation of the inconsistency ranking generated by these inconsistency indices.

Focusing on $TI$ has inspired the idea of a lexicographically optimal completion, which has been introduced recently by \citet{AgostonCsato2023}.

\begin{definition} \label{Def12}
\emph{Lexicographically optimal completion}:
Let $\mathbf{A} \in \mathcal{A}^{\ast}$ be an incomplete pairwise comparison matrix.
Let $\mathbf{A}(\mathbf{x})$ be the pairwise comparison matrix where the missing entries in matrix $\mathbf{A}$ are replaced by variables collected in $\mathbf{x}$. Let $t_{ijk}(\mathbf{x})$ be the inconsistency of the triad determined by the items $1 \leq i,j,k \leq n$ according to the inconsistency index $TI$ in matrix $\mathbf{A}(\mathbf{x})$. 
Let $\theta(\mathbf{x})$ be the vector of the $n(n-1)(n-2)/6$ local inconsistencies $t_{ijk}(\mathbf{x})$ arranged in a non-increasing order: $\theta_u(\mathbf{x}) \geq \theta_v(\mathbf{x})$ for all $u < v$.

Matrix $\mathbf{A}(\mathbf{x})$ is a \emph{lexicographically optimal completion} of the incomplete pairwise comparison matrix $\mathbf{A}$ if there is no lexicographically smaller completion: for any other completion $\mathbf{A}(\mathbf{y})$, there does not exist an index $1 \leq v \leq n(n-1)(n-2)/6$ such that $\theta_u(\mathbf{x}) = \theta_u(\mathbf{y})$ for all $u < v$ and $\theta_v(\mathbf{x}) > \theta_v(\mathbf{y})$.
\end{definition}

Again, the lexicographically optimal completion is unique if the graph associated with the incomplete pairwise comparison is connected \citep[Theorem~1]{AgostonCsato2023}.

\subsection{The research question} \label{Sec24}

The ordinal preferences of the decision-maker are often more reliable than the cardinal values \citep{YuanWuTu2023}. Thus, it is crucial to investigate whether the derived priorities contain any ordinal violation.

\begin{definition} \label{Def13}
\emph{Ordinal violation}:
Let $\mathbf{A} = \left[ a_{ij} \right] \in \mathcal{A}^{\ast}$ be an incomplete pairwise comparison matrix and $\mathbf{w} = \left[ w_i \right] \in \mathbb{R}_+$ be a weight vector.
The weight vector $\mathbf{w}$ shows \emph{ordinal violation} if there exist items $i,j$ such that $a_{ij} > 1$ but $w_i \leq w_j$, or $a_{ij} = 1$ but $w_i \neq w_j$.
\end{definition}

\citet{GolanyKress1993} have suggested the number of violations as an important criterion to compare weighting methods. This issue has been widely discussed in the literature \citep{ChenKouLi2018, CsatoRonyai2016, FaramondiOlivaBozoki2020, Rezaei2015, TuWu2021, TuWu2023, TuWuPedrycz2023, YuanWuTu2023, WangPengKou2021}.

\begin{figure}[t!]
\centering

\begin{tikzpicture}[scale=1,auto=center, transform shape, >=triangle 45]
\tikzstyle{every node}=[draw,align=center];
  \node (N1) at (-2,9) {Connected directed \\ acyclic graph};
  \node (N2) at (-2,6) {Incomplete pairwise \\ comparison matrix $\mathbf{A}$};
  \node (N3) at (-2,3) {(Complete) pairwise \\ comparison matrix $\mathbf{A}(\mathbf{x})$};
  \node (N4) at (-2,0) {Weight vector $\mathbf{w}$};
  \node[shape = ellipse] (N5) at (5,3) {Does the weight vector $\mathbf{w}$ \\ contain ordinal violation with \\ respect to the incomplete \\ pairwise comparison matrix $\mathbf{A}$?};

\tikzstyle{every node}=[align=center];  
  \draw [->,line width=1pt] (N1) -- (N2)  node [midway, left] {Definition~\ref{Def11}};
  \draw [->,line width=1pt] (N2) -- (N3)  node [midway, left] {Completion method};
  \draw [->,line width=1pt] (N3) -- (N4)  node [midway, left] {Weighting method};
  \draw [->,line width=1pt,dashed] (N2) -- (N5)  node [midway, above left] {};
  \draw [->,line width=1pt,dashed] (N4) -- (N5)  node [midway, above left] {};
\end{tikzpicture}
\captionsetup{justification=centering}
\caption{The setting where ordinal violations are investigated}
\label{Fig1}
\end{figure}
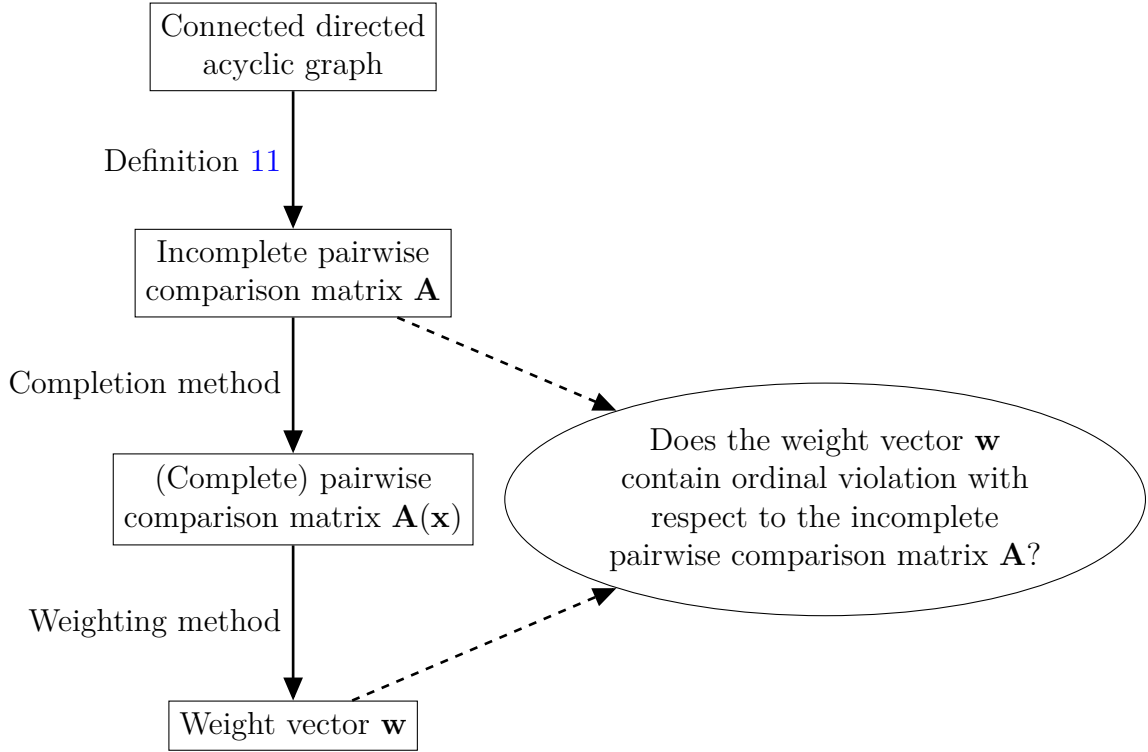

Figure~\ref{Fig1} outlines the problem for which the current paper makes a  substantial contribution. We consider a (connected) directed acyclic graph, construct the associated incomplete pairwise comparison matrix according to Definition~\ref{Def11}, estimate the values of missing entries, and derive a weight vector from the complete pairwise comparison matrix. A natural question is how the lack of ordinal violation can be guaranteed in this process.

\section{The main result} \label{Sec3}

First, two results from the extant literature are recalled to highlight the significance of the problem presented in Figure~\ref{Fig1}.

\begin{lemma} \label{Lemma1}
Let $\mathbf{A} = \left[ a_{ij} \right] \in \mathcal{A}^\ast$ be a CDAG-based incomplete pairwise comparison matrix.
The priorities may contain ordinal violation if the $CR$-optimal completion is used to obtain the missing entries and the eigenvector method is used to derive the weight vector.
\end{lemma}

\begin{proof}
See \citet[Theorem~4.2]{CsatoRonyai2016}.
\end{proof}

\begin{lemma} \label{Lemma2}
Let $\mathbf{A} = \left[ a_{ij} \right] \in \mathcal{A}^\ast$ be a CDAG-based incomplete pairwise comparison matrix.
The priorities may contain ordinal violation if the $GCI$-optimal completion is used to obtain the missing entries and the logarithmic least squares method is used to derive the weight vector.
\end{lemma}

\begin{proof}
See \citet[Theorem~3.3]{CsatoRonyai2016}.
\end{proof}

According to our knowledge, there exists no pair of completion and weighting techniques that guarantees the avoidance of ordinal violations for CDAG-based incomplete pairwise comparison matrices yet. In the following, we provide such a procedure.

\begin{lemma} \label{Lemma3}
Let $\mathbf{A} = \left[ a_{ij} \right] \in \mathcal{A}^{n \times n}$ be a (complete) pairwise comparison matrix and $1 \leq i,j \leq n$ be two items such that $a_{ij} > 1 > a_{ji}$ and $a_{ik} \geq a_{jk}$ for all $k \neq i,j$.
If the weight vector $\mathbf{w} = \left[ w_i \right]$ is derived by the eigenvector method or the logarithmic least squares method, then $w_i > w_j$.
\end{lemma}

\begin{proof}
For the eigenvector method, \eqref{eq_EM} implies that
\[
\lambda_{\max} w_i = \sum_{k=1}^n a_{ik} w_k > \sum_{k=1}^n a_{jk} w_k = \lambda_{\max} w_j.
\]
For the logarithmic least squares method, \eqref{eq_RGM} directly gives the required implication as
\[
\prod_{k=1}^n a_{ik} > \prod_{k=1}^n a_{jk}.
\]
\end{proof}

\begin{theorem} \label{Theo1}
Let $\mathbf{A} = \left[ a_{ij} \right] \in \mathcal{A}^\ast$ be a CDAG-based incomplete pairwise comparison matrix.
The priorities do not contain ordinal violation if the lexicographically optimal completion is used to obtain the missing entries and the eigenvector method or the logarithmic least squares method is used to derive the weight vector.
\end{theorem}

\begin{proof}
For any CDAG-based incomplete pairwise comparison matrix, it can be assumed without loss of generality for all $i < j$ that $a_{ij} = \alpha$ or $a_{ij}$ is missing (due to topological sort) .
Denote by $\mathbf{B} = \left[ b_{ij} \right] \in \mathcal{A}$ the (complete) pairwise comparison matrix obtained from $\mathbf{A}$ by the lexicographically optimal completion. First, it is verified that $a_{ij} = \alpha > 1$ implies $b_{ik} \geq b_{jk}$ for all $k \neq i,j$.

Assume, for contradiction, that $b_{ik} < b_{jk}$ is satisfied for a particular $k$. Consequently, $TI_{ijk} \left( \mathbf{B} \right) > \alpha$ holds for the triad determined by items $i$, $j$, $k$.

Define the (complete) pairwise comparison matrix $\mathbf{C} = \left[ c_{ij} \right] \in \mathcal{A}$ for all $i < j$ as follows:
\begin{itemize}
\item
$c_{ij} = a_{ij}$ if $a_{ij} = \alpha$;
\item
$c_{ij} = \alpha$ if there exists a directed walk (see Definition~\ref{Def7}) from $i$ to $j$ in the underlying connected directed acyclic graph;
\item
$c_{ij} = 1$ otherwise.
\end{itemize}
Entries below the diagonal are filled to satisfy the reciprocity property.

It can be checked that the inconsistency $TI$ of any triad in matrix $\mathbf{C}$ is at most $\alpha$: \eqref{eq_TI} may exceed $\alpha$ only if the $c_{ij} c_{jk} = \alpha^2$ ($c_{ij} c_{jk} = 1 / \alpha^2$), but then there is a directed walk from $i$ to $k$, hence $c_{ik} = \alpha$ ($c_{ik} = 1 / \alpha$).
Therefore, $\mathbf{B}$ is not the lexicographically optimal completion of the incomplete pairwise comparison matrix $\mathbf{A}$ since matrix $\mathbf{C}$ provides a lexicographically smaller completion.

Thus, $a_{ij} = \alpha > 1$ implies $b_{ik} \geq b_{jk}$ for all $k \neq i,j$, and the conditions of Lemma~\ref{Lemma3} are satisfied by the lexicographically optimal completion of the incomplete pairwise comparison matrix $\mathbf{A}$. Hence, $w_i > w_j$ for the weights derived by the eigenvector method or the logarithmic least squares method.
\end{proof}

The key in the proof of Theorem~\ref{Theo1} is the construction of matrix $\mathbf{C}$, which is illustrated below.

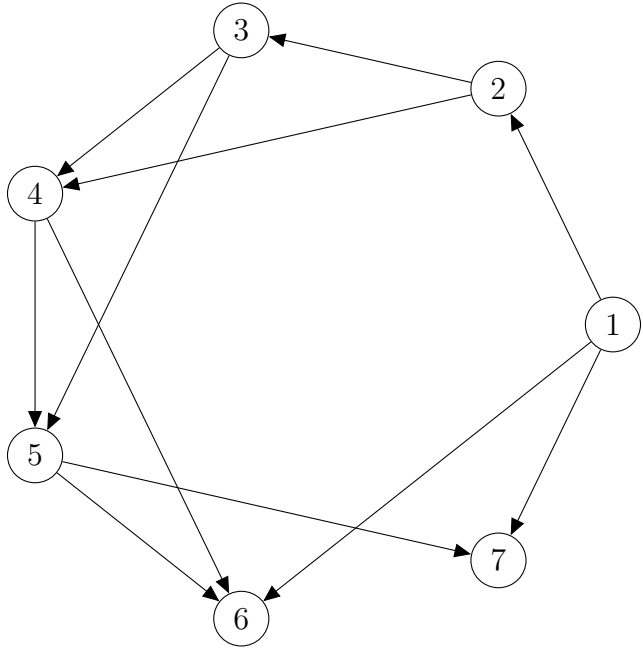
\begin{figure}[t!]
\centering

\begin{tikzpicture}[scale=1, auto=center, transform shape, >=triangle 45]
\tikzstyle{every node}=[draw,shape=circle];
  \node (n1)  at (0:4)  {$1$};
  \node (n2)  at (360/7:4)   {$2$};
  \node (n3)  at (2*360/7:4)   {$3$};
  \node (n4)  at (3*360/7:4)  {$4$};
  \node (n5)  at (4*360/7:4)  {$5$};
  \node (n6)  at (5*360/7:4)  {$6$};
  \node (n7)  at (6*360/7:4)  {$7$};

  \foreach \from/\to in {n1/n2,n1/n6,n1/n7,n2/n3,n2/n4,n3/n4,n3/n5,n4/n5,n4/n6,n5/n6,n5/n7}
    \draw [->] (\from) -- (\to);
\end{tikzpicture}
\caption{The connected directed acyclic graph of Example~\ref{Examp1}}
\label{Fig2}
\end{figure}

\begin{example} \label{Examp1}
Consider the directed acyclic graph shown in Figure~\ref{Fig2} \citep[Example~3.4]{CsatoRonyai2016}.
It is the minimal counterexample that can be used in the proof of Lemma~\ref{Lemma2} regarding the number of items (7), and among them, the number of arcs (11) .
According to Definition~\ref{Def11}, the associated incomplete pairwise comparison matrix $\mathbf{A}$ is as follows:
\[
\mathbf{A}=
\left(
\begin{array}{K{2.5em}K{2.5em}K{2.5em}K{2.5em}K{2.5em}K{2.5em}K{2.5em}}
    1     & \alpha     & \ast & \ast & \ast & \alpha     & \alpha \\
    1/\alpha   & 1     & \alpha     & \alpha     & \ast & \ast & \ast \\
    \ast & 1/\alpha   & 1     & \alpha     & \alpha     & \ast & \ast \\
    \ast & 1/\alpha   & 1/\alpha   & 1     & \alpha     & \alpha     & \ast \\
    \ast & \ast & 1/\alpha   & 1/\alpha   & 1     & \alpha     & \alpha \\
    1/\alpha   & \ast & \ast & 1/\alpha   & 1/\alpha   & 1     & \ast \\
    1/\alpha   & \ast & \ast & \ast & 1/\alpha   & \ast & 1 \\
\end{array}
\right),
\]
where $\alpha > 1$.

This leads to the following (complete) pairwise comparison matrix $\mathbf{C}$:
\[
\mathbf{C}=
\left(
\begin{array}{K{2.5em}K{2.5em}K{2.5em}K{2.5em}K{2.5em}K{2.5em}K{2.5em}}
    1     	 & \alpha   & \alpha   & \alpha   & \alpha   & \alpha & \alpha \\
    1/\alpha & 1     	& \alpha   & \alpha   & \alpha   & \alpha & \alpha \\
    1/\alpha & 1/\alpha & 1        & \alpha   & \alpha   & \alpha & \alpha \\
    1/\alpha & 1/\alpha & 1/\alpha & 1     	  & \alpha   & \alpha & \alpha \\
    1/\alpha & 1/\alpha & 1/\alpha & 1/\alpha & 1     	 & \alpha & \alpha \\
    1/\alpha & 1/\alpha & 1/\alpha & 1/\alpha & 1/\alpha & 1      & 1 \\
    1/\alpha & 1/\alpha & 1/\alpha & 1/\alpha & 1/\alpha & 1	  & 1 \\
\end{array}
\right).
\]
For instance, $c_{15} = \alpha$ due to the directed walk $1 \to 2 \to 3 \to 5$ but $c_{67} = 1$ because there is no directed walk from vertex 6 to vertex 7. It can be seen that $TI \leq \alpha$ for any triad $i,j,k$: while $TI = \alpha$ in the case of most triads, $TI_{167} = TI_{267} = TI_{367} = TI_{467} = TI_{567} = 1$.
\end{example}

The second part of the proof of Theorem~\ref{Theo1} is based on Lemma~\ref{Lemma3}, which is probably satisfied by most (if not any) reasonable weighting methods, similar to the eigenvector and the logarithmic least squares methods.

\section{Conclusions} \label{Sec4}

The current paper has studied an axiom for completion methods of pairwise comparison matrices with missing entries. 
As Figure~\ref{Fig1} shows, a three-step procedure has been investigated:
\begin{enumerate}
\item
A directed acyclic graph is transformed into an incomplete pairwise comparison matrix (Definition~\ref{Def11});
\item
All missing entries of the incomplete pairwise comparison matrix are estimated (Section~\ref{Sec23});
\item
Priorities are determined by using a weighting method (Section~\ref{Sec21}).
\end{enumerate}
In particular, we have focused on the conditions of avoiding ordinal violations in this setting. \citet{CsatoRonyai2016} have already revealed that the $CR$-optimal completion combined with the eigenvector method, as well as the $GCI$-optimal completion combined with the logarithmic least squares method does not satisfy the required property. On the other hand, the lexicographically optimal completion together with any of these popular weighting methods is proven to guarantee the lack of any ordinal violation. This seems to be a strong argument in favour of filling the missing entries by lexicographically minimising the inconsistencies of the triads, which is a recent but promising proposal to handle incomplete pairwise comparisons \citep{AgostonCsato2023}.

There are several interesting directions for future research.
It would be interesting to see other pairs of completion and weighting techniques that avoid ordinal violations. Some restrictions can be imposed on directed acyclic graphs to ensure ordinal consistency for the incomplete eigenvector and the incomplete logarithmic least squares methods. A natural extension of our setting can allow for different intensities of the preferences given by a directed acyclic graph.

\bibliographystyle{apalike}
\bibliography{All_references}

\clearpage
\section*{Appendix}

The lexicographically optimal completion (see Definition~\ref{Def12}) has been introduced and discussed by \citet{AgostonCsato2023}. Here a short overview and an illustrative example are provided on its construction based on \citet[Section~3]{AgostonCsato2023}. Further details can be found in \citet{AgostonCsato2023}.

By regarding the logarithmically transformed entries of the original pairwise comparison matrix, the lexicographically optimal completion can be obtained via solving successive linear programming (LP) problems.

Let $\mathcal{L}$ denote the index set of all triads. The elements of a triad $\ell \in \mathcal{L}$ are denoted by $i_\ell$, $j_\ell$, and $k_\ell$.
Consider the LP problem in the following form:
\begin{align*}
z & \rightarrow \min \tag{LP1.obj} \label{eq: obj_fun_LP} \\
\log a_{i_\ell,j_\ell} +  \log a_{j_\ell,k_\ell} + \log a_{k_\ell,i_\ell} &\le z_\ell  \qquad &\forall \ell\in\mathcal{L} \tag{LP1.1} \label{eq: g_elteres1_LP} \\
\log a_{i_\ell,j_\ell} +  \log a_{j_\ell,k_\ell} + \log a_{k_\ell,i_\ell} &\ge z_\ell  \qquad &\forall \ell\in\mathcal{L} \tag{LP1.2}\label{eq: g_elteres2_LP}
\\
z_\ell &\le z  \qquad &\forall \ell\in\mathcal{L} \tag{LP1.3}\label{eq: prem_LP} \\
\qquad \qquad  z_\ell &\ge 0  \qquad &\forall \ell\in\mathcal{L} \\
z &\ge 0\ \ ,\
\end{align*}
where $\log a_{i_\ell,j_\ell}$, $\log a_{j_\ell,k_\ell}$, $\log a_{k_\ell,i_\ell}$ is a parameter (unbounded decision variable) if the corresponding matrix element is known (missing). The suggested algorithm for the lexicographically optimal completion is provided in Algorithm~\ref{alg_LOC}.

\begin{algorithm}
\caption{Lexicographically optimal completion}
\label{alg_LOC}
\begin{algorithmic}[1]
\State \(\mathcal{M} \leftarrow \mathcal{L}\)
\State solve the LP problem LP1
\State \(obj \leftarrow\) objective value of LP1
\While{\(obj > 0\) }
\State find a constraint \(z_\ell\le z\) (\(\ell\in\mathcal{M}\)) for which the dual variable is negative
\State change constraint \(z_\ell\le z\) to \(z_\ell\le obj\)
\State \(\mathcal{M} \leftarrow \mathcal{M}\setminus\{\ell\}\)
\State solve the modified LP
\State \(obj \leftarrow\) objective value of the modified LP
\EndWhile
\end{algorithmic}
\end{algorithm}

Thus, the lexicographically optimal completion can be obtained by an iterative process:
\begin{enumerate}
\item \label{LP_solution1}
A linear programming problem is solved to minimise the natural triad inconsistency index for all triads with an unknown value of $TI$.
\item
A triad (represented by two constraints in the LP), where the inconsistency index $TI$ cannot be lower, is chosen, which can be seen from the non-zero shadow price of at least one constraint.
\item
The inconsistency index $TI$ is fixed for this triad (or one of these triads if there exists more than one), the associated constraints are removed from the LP, and we return to Step~\ref{LP_solution1}.
\end{enumerate}
The algorithm finishes if the minimal $TI$ is determined for all triads. The number of LPs to be solved is at most the number of triads having an incomplete pairwise comparison (which is finite) because the number of constraints in the LP decreases continuously.

\begin{example} \label{Examp2}
Take the following incomplete pairwise comparison matrix of order four, where $a_{13}$ (thus $a_{31}$) and $a_{14}$ (thus $a_{41}$) remain undefined:
\[
\mathbf{A} = \left[
\begin{array}{K{3em} K{3em} K{3em} K{3em}}
    1     	& a_{12}  	& \ast   	& \ast   \\
    a_{21}	& 1       	& a_{23}	& a_{24} \\
   \ast		& a_{32}	& 1      	& a_{34} \\
   \ast 	& a_{42}  	& a_{43}	& 1 \\
\end{array}
\right].
\]
The four triads imply eight constraints in the first LP due to the reciprocity condition:
\begin{align} \label{LP1_Examp2}
z_1 \to \min \nonumber \\
\log a_{12} + \log a_{23} - \log x_{13} & \leq z_1 \nonumber \\
- \log a_{12} - \log a_{23} + \log x_{13} & \leq z_1 \nonumber \\
\log a_{12} + \log a_{24} - \log x_{14} & \leq z_1 \nonumber \\
- \log a_{12} - \log a_{24} + \log x_{14} & \leq z_1 \nonumber \\
\log x_{13} + \log a_{34} - \log x_{14} & \leq z_1 \nonumber \\
- \log x_{13} - \log a_{34} + \log x_{14} & \leq z_1 \nonumber \\
\log a_{23} - \log a_{24} + \log a_{34} & \leq z_1 \nonumber \\
- \log a_{23} + \log a_{24} - \log a_{34} & \leq z_1
\end{align}

If $a_{12} = 2$, $a_{24} = 8$, $a_{23} = a_{34} = 1$, and the missing elements are substituted by variables, we get the matrix below:
\[
\mathbf{B(\mathbf{x})} = \left[
\begin{array}{K{3em} K{3em} K{3em} K{3em}}
    1     	& 2			& x_{13}   	& x_{14} \\
    1/2		& 1       	& 1			& 8 \\
   1/x_{13}	& 1 		& 1      	& 1 \\
   1/x_{14}	& 1/8	 	& 1  		& 1 \\
\end{array}
\right].
\]
This matrix contains four triads with the following values of $TI$:
\[
TI_{123}(\mathbf{x}) = \max \left\{ \frac{x_{13}}{2}; \frac{2}{x_{13}} \right\};
\]
\[
TI_{124}(\mathbf{x}) = \max \left\{ \frac{x_{14}}{16}; \frac{16}{x_{14}} \right\};
\]
\[
TI_{134}(\mathbf{x}) = \max \left\{ \frac{x_{14}}{x_{13}}; \frac{x_{13}}{x_{14}} \right\};
\]
\[
TI_{234}(\mathbf{x}) = \max \left\{ 8; \frac{1}{8} \right\}.
\]

According to equation (3.6) in \citet{BozokiRapcsak2008}, the Koczkodaj inconsistency index of matrix $\mathbf{B(\mathbf{x})}$ is
\[
KI \left( \mathbf{B(\mathbf{x})} \right) = 1 - \frac{1}{\max \left\{ TI_{123}(\mathbf{x}); TI_{124}(\mathbf{x}); TI_{134}(\mathbf{x}); TI_{234}(\mathbf{x}) \right\}}.
\]
Hence, $KI \left( \mathbf{B(\mathbf{x})} \right)$ is minimal if $t(\mathbf{x}) = \max \left\{ TI_{123}(\mathbf{x}); TI_{124}(\mathbf{x}); TI_{134}(\mathbf{x}); TI_{234}(\mathbf{x}) \right\}$ is minimal. Since $TI_{234}(\mathbf{x}) = 8$, $t(\mathbf{x}) \geq 8$.

In addition, $t(\mathbf{x}) = 8$ if the following conditions are satisfied:
\[
TI_{123}(\mathbf{x}) \leq 8 \iff 1/4 \leq x_{13} \leq 16;
\]
\[
TI_{124}(\mathbf{x}) \leq 8 \iff 2 \leq x_{14} \leq 128;
\]
\[
TI_{134}(\mathbf{x}) \leq 8 \iff 1/8 \leq x_{13} / x_{14} \leq 8.
\]
Therefore, $z_1 = \bar{z_1} = 3$ in \eqref{LP1_Examp2} due to the seventh constraint but there are multiple optimal solutions.

In the second iteration, a simpler LP should be solved after removing the constraints associated with the triad $(2,3,4)$:
\begin{align} \label{LP2_Examp2}
z_2 \to \min \nonumber \\
\log a_{12} + \log a_{23} - \log x_{13} & \leq z_2 \nonumber \\
- \log a_{12} - \log a_{23} + \log x_{13} & \leq z_2 \nonumber \\
\log a_{12} + \log a_{24} - \log x_{14} & \leq z_2 \nonumber \\
- \log a_{12} - \log a_{24} + \log x_{14} & \leq z_2 \nonumber \\
\log x_{13} - \log a_{34} + \log x_{14} & \leq z_2 \nonumber \\
- \log x_{13} + \log a_{34} - \log x_{14} & \leq z_2
\end{align}
\eqref{LP1_Examp2} already has a unique solution: $\log_2 x_{13} = 2$ ($x_{13} = 4$), $\log_2 x_{14} = 3$ ($x_{14} = 8$), $z_2 = 1$.
\end{example}

\end{document}